\documentclass[12pt,letterpaper, reqno]{amsart}
\newcommand{\la}{\langle}
\newcommand{\ra}{\rangle}

\newcommand{\R}{\mathbb{R}}

\usepackage{amssymb}
\usepackage{amsthm}
\usepackage{amsxtra}
\usepackage{graphicx}
\usepackage{mathabx}

\newtheorem{theorem}{Theorem}

\newtheorem{proposition}[theorem]{Proposition}
\newtheorem{lemma}[theorem]{Lemma}
\newtheorem{corollary}[theorem]{Corollary}

\theoremstyle{remark}
\newtheorem{remark}[theorem]{Remark}

\numberwithin{equation}{section}

\numberwithin{theorem}{section}

\numberwithin{table}{section}

\numberwithin{figure}{section}

\ifx\pdfoutput\undefined
  \DeclareGraphicsExtensions{.pstex, .eps}
\else
  \ifx\pdfoutput\relax
    \DeclareGraphicsExtensions{.pstex, .eps}
  \else
    \ifnum\pdfoutput>0
      \DeclareGraphicsExtensions{.pdf}
    \else
      \DeclareGraphicsExtensions{.pstex, .eps}
    \fi
  \fi
\fi

\title[fractional schr\"odinger]{On Fractional Schr\"odinger Equations in sobolev spaces}

\date{\today}
\author{Younghun Hong}
\address{University of Texas, Austin}

\author{Yannick Sire}
\address{Universit\'e Aix-Marseille, I2M, UMR 7353,Marseille, France}

\begin{document}

\maketitle

\begin{abstract}
Let $\sigma\in(0,1)$ with $\sigma\neq\frac{1}{2}$. We investigate the fractional nonlinear Schr\"odinger equation in $\mathbb R^d$:
$$i\partial_tu+(-\Delta)^\sigma u+\mu|u|^{p-1}u=0,\, u(0)=u_0\in H^s,$$
where $(-\Delta)^\sigma$ is the Fourier multiplier of symbol $|\xi|^{2\sigma}$, and $\mu=\pm 1$. This model has been introduced by Laskin in quantum physics \cite{laskin}. We establish local well-posedness and ill-posedness in Sobolev spaces for power-type nonlinearities.  
\end{abstract}

\tableofcontents

\section{Introduction}
Let $\sigma\in(0,1)$ with $\sigma\neq\frac{1}{2}$. We consider the Cauchy problem for the fractional nonlinear Schr\"odinger equation
\begin{equation}\tag{$\textup{NLS}_\sigma$}
i\partial_tu+(-\Delta)^\sigma u+\mu|u|^{p-1}u=0,\ u(0)=u_0\in H^s,
\end{equation}
where $\mu= \pm 1$ depending on the focusing or defocusing case. The operator $(-\Delta)^{\sigma}$ is the so-called fractional laplacian, a Fourier multiplier of $|\xi|^{2\sigma}$. The fractional laplacian is the infinitesimal generator of some Levy processes \cite{B}. A rather extensive study of the potential theoretic aspects of this operator can be found in \cite{landkof}.

The previous equation is a fundamental equation of fractional quantum mechnics, a generalization of the standard quantum mechanics extending the Feynman path integral to Levy processes \cite{laskin}. 

The purpose of the present paper is to develop a general well-posedness and ill-posedness theory in Sobolev spaces. The one-dimensional case has been treated in \cite{CHKL} for cubic nonlinearities, i.e. $p=3$, and $\sigma\in(\frac{1}{2},1)$. Here, we consider a higher-dimensional version and other types of nonlinear terms. We also include all $\sigma\in (0,1)$ except $\sigma=\frac{1}{2}$; furthermore, contrary to \cite{CHKL} where the use of Bourgain spaces was crucial (since the main goal of their paper was to derive well-posedness theory on the flat torus), we rely only on standard Strichartz estimates and functional inequalities in $\R^d.$ In the case of Hartree-type nonlinearities, the local well-posedness and blow-up have been investigated in \cite{ozawa}. 

In the present paper, we will not consider global aspects with large data. For that, we refer the reader to \cite{sire2} for a study of the energy-critical equation in the radial case, following the seminal work of Kenig and Merle \cite{km1,km2}. As a consequence, we do not consider blow-up phenomena, an aspect we will treat in a forthcoming work.

We introduce two important exponents for our purposes: 
$$
s_c=\frac{d}{2}-\frac{2\sigma}{p-1}
$$
and 
$$s_g=\frac{1-\sigma}{2}.$$

Here, $s_c$ is the scaling-critical regularity exponent in the following sense: for $\lambda>0$, the transformation
\[u(t,x)\mapsto \frac{1}{\lambda^{2\sigma/(p-1)}} u\Big(\frac{t}{\lambda^{2\sigma}},\frac{x}{\lambda}\Big), \quad u_0(x)\mapsto \frac{1}{\lambda^{2\sigma/(p-1)}} u_0\Big(\frac{x}{\lambda}\Big)\]
keeps the equation invariant and one can expect local-wellposedness for $s \geq s_c$, since the scaling leaves the $\dot H^{s_c}$ norm invariant. On the other hand, $s_g$ is the critical regularity in the ``pseudo"-Galilean invariance (see the proof of ill-posedness below). 
Under the flow of the equation ($\textup{NLS}_\sigma$), the following
quantities are conserved:
\begin{align*}
M[u]=&\int_{\mathbb R^d} |u(t,x)|^2dx&&\textup{(mass)},\\
E[u]=&\int_{\mathbb R^d}\frac{1}{2}|\,|\nabla|^{\sigma}u(t,x)|^2+\frac{\mu
}{p+1}|u(t,x)|^{p+1}dx.&&\textup{(energy)}.
\end{align*}

An important feature of the equation under study is a loss of derivatives for the Strichartz estimates as proved in \cite{COX}. Unless additional assumptions are met such as radiality as in \cite{zihua}, one has a loss of $d(1-\sigma)$ derivatives in the dispersion (see \eqref{dispersive estimate with loss}). This happens to be an issue in several arguments.

\subsection*{Main results} The goal of this paper is to show that $(\textup{NLS}_\sigma)$ is locally well-posed in $H^s$ for $s\geq \max(s_c,s_g, 0)$, and it is ill-posed in $H^s$ for $s\in (s_c, 0)$. We start with well-posedness results. 
\begin{theorem}[Local well-posedness in subcritical cases]\label{subcritical LWP}
Let
\begin{align*}
\left\{\begin{aligned}
&s\geq s_g&&\text{ when }d=1\textup{ and }2\leq p<5,\\
&s>s_c&&\text{ when }d=1\textup{ and }p\geq 5,\\
&s>s_c&&\text{ when }d\geq 2\textup{ and }p\geq3.
\end{aligned}
\right.
\end{align*}
Then, $(\textup{NLS}_\sigma)$ is locally well-posed in $H^s$. 
\end{theorem}

\begin{theorem}[Local well-posedness in critical cases]\label{critical LWP}
Suppose that
\begin{align*}
\left\{\begin{aligned}
&p>5&&\text{ when }d=1,\\
&p>3&&\text{ when }d\geq 2.
\end{aligned}
\right.
\end{align*}
Then, $(\textup{NLS}_\sigma)$ is locally well-posed in $H^{s_c}$. 
\end{theorem}

The proof of Theorem \ref{critical LWP} is based on a new method, improving on estimates in \cite{CKSTT}. This improvement, based on controlling the nonlinearity in a suitable space, is necessary due to the loss of derivatives in the Strichartz estimates. 

As a by-product, we also prove small data scattering. 
\begin{theorem}[Small data scattering]\label{scattering}
Suppose that
\begin{align*}
\left\{\begin{aligned}
&p>5&&\text{ when }d=1,\\
&p>3&&\text{ when }d\geq 2.
\end{aligned}
\right.
\end{align*}
Then, there exists $\delta>0$ such that if $\|u_0\|_{H^{s_c}}<\delta$, then $u(t)$ scatters in $H^{s_c}$. Precisely, there exist $u_\pm \in H^{s_c}$ such that
$$\lim_{t\to\pm\infty}\|u(t)-e^{it(-\Delta)^\sigma}u_\pm\|_{H^{s_c}}=0.$$
\end{theorem}

\begin{remark}
Contrary to the case $\sigma\neq\frac{1}{2}$, when $\sigma=\frac{1}{2}$, the fractional NLS does not have small data scattering. See \cite{krieger}. 
\end{remark}

Finally, our last theorem is the ill-posedness result. Note that our result is not optimal, since one should expect ill-posedness in $H^s$ up to $s_g=\frac{1-\sigma}{2}$, which is nonnegative. We hope to come back to this issue in a forthcoming work. 

\begin{theorem}[Ill-posedness]\label{ill-posedness}
Let $d=1,2$ or $3$ and $\sigma\in(\frac{d}{4},1)$. If $p$ is not an odd integer, we further assume that $p\geq k+1$, where $k$ is an integer larger than $\frac{d}{2}$. Then, $(\textup{NLS}_\sigma)$ is ill-posed in $H^s$ for $s\in (s_c,0)$.
\end{theorem}

An interesting feature of the previous ill-posedness result is the fact that, contrary to the standard NLS equation ($\sigma=1$) there is no exact Galilean invariance. However, one can introduce a new ``pseudo-Galilean invariance" which is enough to our purposes. More precisely, for $v\in\mathbb{R}^d$, we define the  transformation
$$\mathcal{G}_vu(t,x)=e^{-iv\cdot x}e^{it|v|^{2\sigma}}u(t,x-2t\sigma|v|^{2(\sigma-1)}v).$$
Note that when $\sigma=1$, $\mathcal{G}_v$ is simply a Galilean transformation, and that NLS is invariant under this transformation, that is, if $u(t)$ solves NLS, so does $\mathcal{G}_vu(t)$. However, when $\sigma\neq1$, $(\textup{NLS}_\sigma)$ is not exactly symmetric with respect to pseudo-Galilean transformations. This opens the construction of solitons for $(\textup{NLS}_\sigma)$ which happen to be different from the ones constructed in the standard case $\sigma=1$. Indeed, if we search for exact solutions of the type
\begin{equation}
u(t,x)=e^{it(|v|^{2\sigma}-\omega^{2\sigma})} e^{-iv\cdot x}Q_\omega(x-2t\sigma|v|^{2(\sigma-1)}v),
\end{equation}
then the profile $Q_\omega$ solves the pseudo-differential equation
\begin{equation}\label{Q equation}
\mathcal{P}_vQ_\omega+\omega^{2\sigma}Q_\omega-|Q_\omega|^{p-1}Q_\omega=0,
\end{equation}
where
\begin{equation}
\mathcal{P}_v=e^{iv\cdot x}(-\Delta)^\sigma e^{-iv\cdot x}-|v|^{2\sigma}-2i\sigma|v|^{2\sigma-2}v\cdot\nabla,
\end{equation}
i.e., $\mathcal{P}_v$ is a Fourier multiplier $\widehat{\mathcal{P}_v f}(\xi)=p_v(\xi)\hat{f}(\xi)$, wiht symbol
\begin{equation}
p_v(\xi)=|\xi-v|^{2\sigma}-|v|^{2\sigma}+2\sigma|v|^{2\sigma-2}v\cdot\xi.
\end{equation}
We plan to come back to this issue in future works.

\section{Strichartz Estimates}

In this section, we review Strichartz estimates for the linear fractional Schr\"odinger operators. 
We say that $(q,r)$ is \textit{admissible} if
$$\frac{2}{q}+\frac{d}{r}=\frac{d}{2},\quad 2\leq q,r\leq\infty,\quad (q,r,d)\neq(2,\infty,2).$$
We define the Strichartz norm by
$$\|u\|_{S_{q,r}^s(I)}:=\||\nabla|^{-d(1-\sigma)(\frac{1}{2}-\frac{1}{r})}u\|_{L_{t\in I}^qW_x^{s,r}},$$
where $I=[0,T)$. Let $\psi: \mathbb{R}^d\to [0,1]$ be a compactly supported smooth function such that $\sum_{N\in 2^{\mathbb{Z}}}\psi_N=1$, where $\psi_N(\xi)=\psi(\frac{\xi}{N})$. For dyadic $N\in 2^{\mathbb{Z}}$, let $P_N$ be a Littlewood-Paley projection, that is, $\widehat{P_N f}(\xi)=\psi(\frac{\xi}{N})\hat{f}(\xi)$.  Then, we define a slightly stronger Strichartz norm by 
$$\|u\|_{\tilde{S}_{q,r}^s(I)}:=\Big(\sum_{N\in2^\mathbb{Z}}\|P_N(|\nabla|^{-d(1-\sigma)(\frac{1}{2}-\frac{1}{r})}u)\|_{L_{t\in I}^qW_x^{s,r}}^2\Big)^{1/2}.$$

\begin{proposition}[Strichartz estimates \cite{COX}]\label{Strichartz}
For an admissible pair $(q,r)$, we have
\begin{align*}
\|e^{it(-\Delta)^\sigma}u_0\|_{S_{q,r}^s(I)}, \|e^{it(-\Delta)^\sigma}u_0\|_{\tilde{S}_{q,r}^s(I)}&\lesssim\|u_0\|_{H^s},\\
\Big\|\int_0^t e^{i(t-s)(-\Delta)^\sigma}F(s)ds\Big\|_{S_{q,r}^s(I)}&\lesssim \|F\|_{L_{t\in I}^1H_x^{s}},\\
\Big\|\int_0^t e^{i(t-s)(-\Delta)^\sigma}F(s)ds\Big\|_{\tilde{S}_{q,r}^s(I)}&\lesssim \|F\|_{L_{t\in I}^1H_x^{s}}.
\end{align*}
\end{proposition}

\begin{proof}[Sketch of Proof]
By the standard stationary phase estimate, one can show that
$$\|e^{it(-\Delta)^\sigma}P_1\|_{L^1\to L^\infty}\lesssim|t|^{-\frac{d}{2}},$$
and by scaling,
\begin{equation}\label{dispersive estimate with loss}
\|e^{it(-\Delta)^\sigma}P_N\|_{L^1\to L^\infty}\lesssim N^{d(1-\sigma)}|t|^{-\frac{d}{2}}.
\end{equation}
Then, it follows from the argument of Keel-Tao \cite{KT} that for any $I\subset\mathbb{R}$,
\begin{align*}
\|e^{it(-\Delta)^\sigma}P_N(|\nabla|^{-d(1-\sigma)(\frac{1}{2}-\frac{1}{r})}u_0)\|_{L_{t\in I}^q W_x^{s,r}}&\lesssim\|P_Nu_0\|_{H^s},\\
\Big\|\int_0^t e^{i(t-s)(-\Delta)^\sigma}P_N(|\nabla|^{-d(1-\sigma)(\frac{1}{2}-\frac{1}{r})}F)(s)ds\Big\|_{L_{t\in I}^qW_x^{s,r}}&\lesssim \|P_NF\|_{L_{t\in I}^1H_x^{s}}.
\end{align*}
Squaring the above inequalities and summing them over all dyadic numbers in $2^{\mathbb{Z}}$, we prove Strichartz estimates.
\end{proof}

The loss of derivatives is due to the Knapp phenomenon (see \cite{zihua}). However, in the radial case, one can overcome this loss as proved in \cite{zihua}, restricting then the admissible powers of the fractional laplacian. Indeed, in \cite{zihua}, this is proved that one has optimal Strichartz estimates if $\sigma \in (d/(2d-1),1)$. In particular, the number $d/(2d-1)$ is larger than $1/2$ and there is a gap between the Strichartz estimates for the wave operator $\sigma=1/2$ and the one occuring for higher powers. This issue suggests that a new phenomenon might occur for this range of powers.

\section{Local Well-posedness}

We establish local well-posedness of the fractional NLS by the standard contraction mapping argument based on Strichartz estimates. Due to loss of regularity in Strichartz estimates, our proof relies on the $L_x^\infty$ bounds (see Lemma 3.2 and 3.3).

\subsection{Subcritical cases}

First, we consider the case that $d=1$ and $2\leq p<5$. In this case, the equation is scaling-subcritical in $H^s$ for $s>s_g$, since $s_g>s_c$. We remark that in the proof, we control the $L_{t\in I}^4L_x^\infty$ norm simply by Strichartz estimates (see \eqref{1d Strichartz 1} and \eqref{1d Strichartz 2}).

\begin{proof}[Proof of Theorem \ref{subcritical LWP} when $d=1$ and $2\leq p<5$]
We define
$$\Phi_{u_0}(u):=e^{it(-\Delta)^\sigma}u_0+ i\mu\int_0^t e^{i(t-s)(-\Delta)^\sigma}(|u|^{p-1}u)(s)ds.$$
Let
$$\|u\|_{X^s}:=\|u\|_{L_{t\in I}^\infty H_x^s\cap L_{t\in I}^4 L_x^\infty},$$
where $I=[0,T)$. Then, applying the 1d Strichartz estimates
\begin{align}
\|e^{it(-\Delta)^\sigma}u_0\|_{L_{t\in I}^4 L_x^\infty}&\lesssim\|u_0\|_{\dot{H}^{s_g}}\label{1d Strichartz 1},\\
\|e^{it(-\Delta)^\sigma}u_0\|_{L_{t\in I}^\infty H_x^s}&\lesssim\|u_0\|_{H^s},\nonumber\\
\Big\|\int_0^t e^{i(t-s)(-\Delta)^\sigma}F(s)ds\Big\|_{L_{t\in I}^4L_x^\infty}&\lesssim \|F\|_{L_{t\in I}^1\dot{H}_x^{s_g}}\label{1d Strichartz 2},\\
\Big\|\int_0^t e^{i(t-s)(-\Delta)^\sigma}F(s)ds\Big\|_{L_{t\in I}^\infty H_x^s}&\lesssim \|F\|_{L_{t\in I}^1H_x^s}\nonumber,
\end{align}
we get
$$\|\Phi_{u_0}(u)\|_{X^s}\lesssim \|u_0\|_{H^s}+\||u|^{p-1}u\|_{L_{t\in I}^1H_x^s}.$$
By the fractional chain rule
\begin{equation}\label{chain rule}
\||\nabla|^s F(u)\|_{L^q}\lesssim \|F'(u)\|_{L^{p_1}}\||\nabla|^s u\|_{L^{p_2}},
\end{equation}
where $s>0$ and $\frac{1}{q}=\frac{1}{p_1}+\frac{1}{p_2}$, and H\"older inequality, we obtain
$$\||u|^{p-1}u\|_{L_{t\in I}^1H_x^s}\lesssim \Big\| \||u|^{p-1}\|_{L_x^\infty}\|u\|_{H_x^s}\Big\|_{L_{t\in I}^1} \leq T^{\frac{5-p}{4}}\|u\|_{L_{t\in I}^4L_x^\infty}^{p-1}\|u\|_{L_{t\in I}^\infty H_x^s}.$$
For the fractional chain rule \eqref{chain rule}, we refer \cite{CW}, for example. We remark that one can choose $p_1=\infty$ in \eqref{chain rule}. Indeed, this can be proved by a little modification of the last step in the proof of Proposition 3.1 in \cite{CW}. Thus, we have
$$\|\Phi_{u_0}(u)\|_{X^s}\lesssim \|u_0\|_{H^s}+T^{\frac{5-p}{4}} \|u\|_{X^s}^p.$$
Similarly, by Strichartz estimates,
$$\|\Phi_{u_0}(u)-\Phi_{u_0}(v)\|_{X^s}\lesssim \||u|^{p-1}u-|v|^{p-1}v\|_{L_{t\in I}^1H_x^s}.$$
Then, applying the fractional Leibniz rule and the fractional chain rule in \cite{CW}, we get 
\begin{align*}
\||u|^{p-1}u-|v|^{p-1}v\|_{H_x^s}&=\Big\|\int_0^1 p|v+t(u-v)|^{p-1}(u-v) dt\Big\|_{H_x^s}\\
&\leq p\int_0^1\||v+t(u-v)|^{p-1}(u-v) \|_{H_x^s}dt\\
&\lesssim\int_0^1\|v+t(u-v)\|_{L_x^\infty}^{p-1}\|u-v\|_{H_x^s}\\
&\ \ \ +\||v+t(u-v)|^{p-1}\|_{H_x^s}\|u-v\|_{L_x^\infty}dt\\
&\lesssim\int_0^1\|v+t(u-v)\|_{L_x^\infty}^{p-1}\|u-v\|_{H_x^s}\\
&\ \ \ +\|v+t(u-v)\|_{L_x^\infty}^{p-2}\|v+t(u-v)\|_{H_x^s}\|u-v\|_{L_x^\infty}dt\\
&\leq (\|u\|_{L_x^\infty}^{p-1}+\|v\|_{L_x^\infty}^{p-1})\|u-v\|_{H_x^s}\\
&\ \ \ +(\|u\|_{L_x^\infty}^{p-2}+\|v\|_{L_x^\infty}^{p-2})(\|u\|_{H_x^s}+\|v\|_{H_x^s})\|u-v\|_{L_x^\infty}.
\end{align*}
Thus, it follows that
\begin{align*}
&\|\Phi_{u_0}(u)-\Phi_{u_0}(v)\|_{X^s}\\
&\lesssim T^{\frac{5-p}{4}}\Big\{(\|u\|_{L_{t\in I}^4L_x^\infty}^{p-1}+\|v\|_{L_{t\in I}^4L_x^\infty}^{p-1})\|u-v\|_{L_{t\in I}^\infty H_x^s}\\
&\quad\quad\quad+(\|u\|_{L_{t\in I}^4L_x^\infty}^{p-2}+\|v\|_{L_{t\in I}^4L_x^\infty}^{p-2})(\|u\|_{L_{t\in I}^\infty H_x^s}+\|v\|_{L_{t\in I}^\infty H_x^s})\|u-v\|_{L_{t\in I}^4L_x^\infty}\Big\}\\
&\lesssim T^{\frac{5-p}{4}} (\|u\|_{X^s}^{p-1}+\|v\|_{X^s}^{p-1})\|u-v\|_{X^s}.
\end{align*}
Choosing sufficiently small $T>0$, we conclude that $\Phi_{u_0}$ is a contraction on a ball
$$B:=\{u: \|u\|_{X^s}\leq 2\|u_0\|_{H^s}\}$$
equipped with the norm $\|\cdot\|_{X^s}$.
\end{proof}

Next, we will prove Theorem \ref{subcritical LWP} when $d=1$ and $p\geq 5$, or $d\geq 2$ and $p\geq3$. In this case, we do not have a good control on the $L_x^\infty$ norm from Strichartz estimates. Instead, we make use of Sobolev embedding.

\begin{lemma}[$L_{t\in I}^{p-1}L_x^\infty$ bound]
Suppose that $d=1$ and $p\geq5$, or $d\geq 2$ and $p\geq3$. Let $s>s_c$. Then, we have
\begin{equation}\label{subcritical Lx-infty bound}
\|u\|_{L_{t\in I}^{p-1}L_x^\infty}\lesssim T^{0+}\|u\|_{S_{q_0,r_0}^{s}(I)},
\end{equation}
where $(q_0,r_0) =((p-1)^+,\Big(\tfrac{2d(p-1)}{d(p-1)-4}\Big)^-)$ is an admissible pair. Here, we denote by $c^+$ a number larger than $c$ but arbitrarily close to $c$, and similarly for $c^-$.
\end{lemma}

\begin{proof}
We observe that
$$\frac{1}{r_0}-\frac{s-d(1-\sigma)(\frac{1}{2}-\frac{1}{r_0})}{d}<0.$$
Thus, by Sobolev inequality,
$$\|u\|_{L_{t\in I}^{p-1}L_x^\infty}\leq T^{0+}\|u\|_{L_{t\in I}^{q_0}L_x^\infty}\lesssim\||\nabla|^{-d(1-\sigma)(\frac{1}{2}-\frac{1}{r_0})}u\|_{L_{t\in I}^{q_0}W_x^{s,r_0}}=\|u\|_{S_{q_0,r_0}^s(I)}.$$
\end{proof}

We also employ a standard persistence of regularity argument.
\begin{lemma}[Persistence of regularity]\label{persistence of regularity}
Let $1<q\leq\infty$, $1<r<\infty$ and $s_1\geq s_2$. Then, $B=\{u: \|u\|_{L_{t\in I}^q W_x^{s_1,r}}\leq R\}$, equipped with the norm $\|\cdot\|_{L_{t\in I}^q W_x^{s_2,r}}$, is a complete metric space.
\end{lemma}

\begin{proof}
We recall:
\begin{theorem}[Theorem 1.2.5 in \cite{Caz}]Consider two Banach spaces $X\hookrightarrow Y$ and $1<p,q\leq \infty$. Let $(f_n)_{n\geq 0}$ be a bouned sequence in $L^q(I, Y)$ and let $f: I\to Y$ be such that $f_n(t)\rightharpoonup f(t)$ in $Y$ as $n\to\infty$, for a.e. $t\in I$. If $(f_n)_{n\geq 0}$ is bounded in $L^p(I;X)$ and if $X$ is reflexive, then $f\in L^p(I;X)$ and $\|f\|_{L^p(I;X)}\leq \liminf_{n\to\infty}\|f_n\|_{L^p(I;X)}$.
\end{theorem}

Suppose that $(f_n)_{n=1}^\infty$ be a Cauchy sequence in $B$. Then, $f_n$ converges to $f$ in $L_{t\in I}^q W_x^{s_2,r}$. Moreover, it follows from Theorem 1.2.5 in \cite{Caz} that
$$\|f\|_{L_{t\in I}^q W_x^{s_1,r}}\leq \liminf_{n\to\infty}\|f_n\|_{L_{t\in I}^q W_x^{s_1,r}}\leq R,$$
and thus $f\in B$. Therefore, we conclude that $B$ is complete.
 \end{proof}

\begin{proof}[Proof of Theorem \ref{subcritical LWP} when $d=1$ and $p\geq 5$, or $d\geq 2$ and $p\geq3$]
Define the map $\Phi_{u_0}$ as above, and let 
$$X^\alpha:=L_{t\in I}^\infty H_x^\alpha\cap S_{q_0,r_0}^\alpha(I),$$
where $(q_0, r_0)$ is an admissible pair in Lemma 3.2. Then, by Strichartz estimates, the fractional chain rule and \eqref{subcritical Lx-infty bound}, we get
\begin{align*}
\|\Phi_{u_0}(u)\|_{X^s}&\lesssim \|u_0\|_{H^s}+\||u|^{p-1}u\|_{L_{t\in I}^1H_x^s}\\
&\lesssim \|u_0\|_{H^s}+\|u\|_{L_{t\in I}^{p-1}L_x^\infty}^{p-1}\|u\|_{L_{t\in I}^\infty H_x^s}\\
&\lesssim \|u_0\|_{H^s}+T^{0+}\|u\|_{S_{q_0,r_0}^{s}(I)}^{p-1}\|u\|_{L_{t\in I}^\infty H_x^s}\\
&\leq \|u_0\|_{H^s}+T^{0+}\|u\|_{X^s}^p,
\end{align*}
and similarly
\begin{equation}\label{difference in LWP proof}
\begin{aligned}
\|\Phi_{u_0}(u)-\Phi_{u_0}(v)\|_{X^0}&\lesssim\||u|^{p-1}u-|v|^{p-1}v\|_{L_{t\in I}^1L_x^2}\\
&\lesssim\|(|u|^{p-1}+|v|^{p-1})|u-v|\|_{L_{t\in I}^1L_x^2}\\
&\lesssim(\|u\|_{L_{t\in I}^{p-1}L_x^\infty}^{p-1}+\|v\|_{L_{t\in I}^{p-1}L_x^\infty}^{p-1})\|u-v\|_{L_{t\in I}^\infty L_x^2}\\
&\lesssim T^{0+}(\|u\|_{S_{q_0,r_0}^{s}(I)}^{p-1}+\|v\|_{S_{q_0,r_0}^{s}(I)}^{p-1})\|u\|_{L_{t\in I}^\infty L_x^2}\\
&\lesssim T^{0+}(\|u\|_{X^s}^{p-1}+\|v\|_{X^s}^{p-1})\|u-v\|_{X^0}.
\end{aligned}
\end{equation}
Thus, for sufficiently small $T>0$, $\Phi_{u_0}$ is contractive on a ball
$$B:=\{u: \|u\|_{X^s}\leq 2\|u_0\|_{H^s}\}$$
equipped with the norm $\|\cdot\|_{X^0}$, which is complete by Lemma \ref{persistence of regularity}.
\end{proof}

\begin{remark}
The standard persistence of regularity argument allows us to avoid derivatives in \eqref{difference in LWP proof}. Indeed, for $u\in B$, $\|\la\nabla\ra^su\|_{L_{t\in I}^{p-1}L_x^\infty}$ is not necessarily bounded.
\end{remark}

\subsection{Scaling-critical cases}
In the scaling-critical case, we use the following lemma, which plays the same role as \eqref{subcritical Lx-infty bound}. We note that the norms in the lemma are defined via the Littlewood-Paley projection in order to overcome the failure of the Sobolev embedding $W^{s,p}\hookrightarrow L^q$, $\frac{1}{q}=\frac{1}{p}-\frac{s}{d}$, when $q=\infty$. Lemma 3.3 generalizes \cite[Lemma 3.1]{CKSTT}.

\begin{lemma}[Scaling-critical $L_{t\in I}^{p-1}L_x^\infty$ bound]
\begin{equation}\label{critical Lx-infty bound}
\|u\|_{L_{t\in I}^{p-1}L_x^\infty}^{p-1}\lesssim \left\{\begin{aligned}
&\|u\|_{\tilde{S}_{4,\infty}^{s_c}(I)}^4\|u\|_{\tilde{S}_{\infty,2}^{s_c}(I)}^{p-5}&&\text{ when }d=1\textup{ and }p>5,\\
&\|u\|_{\tilde{S}_{2+,\infty-}^{s_c}(I)}^{2}\|u\|_{\tilde{S}_{\infty,2}^{s_c}(I)}^{p-3}&&\text{ when }d=2\textup{ and }p>3,\\
&\|u\|_{\tilde{S}_{2,\frac{2d}{d-2}}^{s_c}(I)}^{2}\|u\|_{\tilde{S}_{\infty,2}^{s_c}(I)}^{p-3}&&\text{ when }d\geq 3\textup{ and }p>3.
\end{aligned}
\right.
\end{equation}
\end{lemma}

\begin{proof}
We will prove the lemma only when $d\geq3$. By interpolation $\|f\|_{L^{p_\theta}}\leq \|f\|_{L^{p_0}}^\theta \|f\|_{L^{p_1}}^{1-\theta}$, $\frac{1}{p_\theta}=\frac{\theta}{p_0}+\frac{1-\theta}{p_1}$, $0<\theta<1$, it suffices to show the lemma for rational $(p-1)=\frac{m}{n}>2$ with $\gcd(m,n)=1$. First, we estimate
$$A(t)=\Big[\sum_N\|P_{N}u(t)\|_{L_x^\infty}\Big]^m\sim\sum_{N_1\geq\cdots\geq N_m}\prod_{i=1}^m\|P_{N_i}u(t)\|_{L_x^\infty}.$$
Observe from Bernstein's inequality that 
\begin{align}
\|P_Nu(t)\|_{L_x^\infty}&\lesssim N^{-\frac{\sigma(p-3)}{p-1}}d_N\label{Bernstein1},\\
\|P_Nu(t)\|_{L_x^\infty}&\lesssim N^{\frac{2\sigma}{p-1}}d_N'\label{Bernstein2},
\end{align}
where
$$d_N=\|P_Nu(t)\|_{\dot{W}_x^{s_c-(1-\sigma),\frac{2d}{d-2}}},\ d_N'=\|P_Nu(t)\|_{\dot{H}_x^{s_c}}.$$
As a consequence, we have
\begin{equation}\label{Bernstein3}
\|P_Nu(t)\|_{L_x^\infty}\lesssim\Big(N^{-\frac{\sigma(p-3)}{p-1}}d_N\Big)^\theta\Big(N^{\frac{2\sigma}{p-1}}d_N'\Big)^{1-\theta}=N^{\frac{\sigma(p-3)}{(p-1)(p-2)}}(d_N)^\theta (d_N')^{1-\theta},
\end{equation}
where $\theta=\frac{1}{p-2}$. Hence, applying \eqref{Bernstein1} for $i=1,\cdots,n$ and \eqref{Bernstein3} for $i=n+1,\cdots,m$, we bound $A(t)$ by
$$\lesssim\sum_{N_1\geq \cdots\geq N_m} \Big(\prod_{i=1}^nN_i^{-\frac{\sigma(p-3)}{p-1}}d_{N_i}\Big) \Big(\prod_{i=n+1}^m N_i^{\frac{\sigma(p-3)}{(p-1)(p-2)}}(d_{N_i})^\theta (d_{N_i}')^{1-\theta}\Big).$$
For an arbitrarily small $\epsilon>0$, we let
$$\tilde{d}_N=\sum_{N'\in 2^{\mathbb{Z}}} \min\Big(\frac{N}{N'},\frac{N'}{N}\Big)^\epsilon d_{N'},\quad \tilde{d}_N'=\sum_{N'\in 2^{\mathbb{Z}}} \min\Big(\frac{N}{N'},\frac{N'}{N}\Big)^\epsilon d_{N'}'.$$
Then, since $d_N\leq \tilde{d}_N$ and $\tilde{d}_{N_i}\leq (\frac{N_1}{N_{i}})^\epsilon \tilde{d}_{N_1}$ and similarly for primes, $A(t)$ is bounded by
$$\lesssim\sum_{N_1\geq \cdots\geq N_m} \Big(\prod_{i=1}^n N_i^{-\frac{\sigma(p-3)}{p-1}}\Big(\frac{N_1}{N_{i}}\Big)^\epsilon \tilde{d}_{N_1}\Big) \Big(\prod_{i=n+1}^m N_i^{\frac{\sigma(p-3)}{(p-1)(p-2)}}\Big(\frac{N_1}{N_{i}}\Big)^\epsilon (\tilde{d}_{N_1})^\theta (\tilde{d}_{N_1}')^{1-\theta}\Big).
$$
Summing in $N_m, N_{m-1}, ..., N_{n+1}$ and using $m-n=(p-2)n$,
\begin{align*}
A(t)&\lesssim\sum_{N_1\geq \cdots\geq N_n} \Big(\prod_{i=1}^n N_i^{-\frac{\sigma(p-3)}{p-1}}\Big(\frac{N_1}{N_{i}}\Big)^\epsilon \tilde{d}_{N_1}\Big)\\
&\quad\quad\quad\quad\times N_n^{\frac{\sigma(p-3)(m-n)}{(p-1)(p-2)}}\Big(\frac{N_1}{N_{n}}\Big)^{(m-n)\epsilon} (\tilde{d}_{N_1})^{(m-n)\theta} (\tilde{d}_{N_1}')^{(m-n)(1-\theta)}\\
&=\sum_{N_1\geq \cdots\geq N_n} \Big(\prod_{i=1}^n N_i^{-\frac{\sigma(p-3)}{p-1}}\Big(\frac{N_1}{N_{i}}\Big)^\epsilon \tilde{d}_{N_1}\Big)\\
&\quad\quad\quad\quad\times N_n^{\frac{\sigma(p-3)n}{p-1}}\Big(\frac{N_1}{N_{n}}\Big)^{(p-2)n\epsilon} (\tilde{d}_{N_1})^{(m-n)\theta} (\tilde{d}_{N_1}')^{(m-n)(1-\theta)},
\end{align*}
and then summing in $N_n, N_{n-1}, ..., N_1$, we obtain that
$$A(t)\lesssim\sum_{N_1} (\tilde{d}_{N_1})^{n+(m-n)\theta} (\tilde{d}_{N_1}')^{(m-n)(1-\theta)}=\sum_{N_1} (\tilde{d}_{N_1})^{2n} (\tilde{d}_{N_1}')^{m-2n},$$
which is, by H\"older inequality and Young's inequality, bounded by
\begin{align*}
&\lesssim\|(\tilde{d}_{N})^{2n} \|_{\ell_N^2}\|(\tilde{d}_{N}')^{m-2n}\|_{\ell_N^{2}}=\|\tilde{d}_{N}\|_{\ell_N^{4n}}^{2n}\|\tilde{d}_{N}'\|_{\ell_N^{2(m-2n)}}^{m-2n}\\
&\leq \|\tilde{d}_{N}\|_{\ell_N^2}^{2n}\|\tilde{d}_{N}'\|_{\ell_N^2}^{m-2n}\lesssim \|d_{N}\|_{\ell_N^2}^{2n}\|d_{N}'\|_{\ell_N^2}^{m-2n}=\|d_{N}\|_{\ell_N^2}^{2n}\|d_{N}'\|_{\ell_N^2}^{(p-3)n}.
\end{align*}
Finally, by the estimate for $A(t)$, we prove that 
\begin{align*}
\|u\|_{L_{t\in I}^{p-1}L_x^\infty}^{p-1}&\leq \int_I A(t)^{p-1}dt=\int_I A(t)^\frac{m}{n}dt\\
&\lesssim\int_I \|d_N\|_{\ell_N^2}^2\|d_N'\|_{\ell_N^2}^{p-3} dt\leq \|d_N\|_{L_{t\in I}^2\ell_N^2}^2\|d_N'\|_{L_{t\in I}^\infty\ell_N^2}^{p-3}\\
&\leq \|u\|_{\tilde{S}_{2,\frac{2d}{d-2}}^{s_c}(I)}^{2}\|u\|_{\tilde{S}_{\infty,2}^{s_c}(I)}^{p-3}.
\end{align*}
\end{proof}

\begin{proof}[Proof of theorem \ref{critical LWP}]
For simplicity, we assume that $d\geq3$. Indeed, with little modifications, we can prove the theorem when $d=1,2$. We define $\Phi_{u_0}(u)$ as in the proof of Theorem \ref{subcritical LWP}. Then, by Strichartz estimates, the fractional chain rule and \eqref{critical Lx-infty bound}, we have
\begin{align*}
\|\Phi_{u_0}(u)\|_{\tilde{S}_{2,\frac{2d}{d-2}}^{s_c}(I)}&\leq \|e^{it(-\Delta)^\sigma}u_0\|_{\tilde{S}_{2,\frac{2d}{d-2}}^{s_c}(I)}+c_0\||u|^{p-1}u\|_{L_{t\in I}^1H^{s_c}}\\
&\leq \|e^{it(-\Delta)^\sigma}u_0\|_{\tilde{S}_{2,\frac{2d}{d-2}}^{s_c}(I)}+c_1\|u\|_{L_{t\in I}^{p-1}L_x^\infty}^{p-1}\|u\|_{L_{t\in I}^\infty H^{s_c}}\\
&\leq \|e^{it(-\Delta)^\sigma}u_0\|_{\tilde{S}_{2,\frac{2d}{d-2}}^{s_c}(I)}+c\|u\|_{\tilde{S}_{2,\frac{2d}{d-2}}^{s_c}(I)}^2\|u\|_{\tilde{S}_{\infty,2}^{s_c}(I)}^{p-2}.
\end{align*}
Similarly, one can show that
\begin{align*}
\|\Phi_{u_0}(u)\|_{\tilde{S}_{\infty,2}^{s_c}(I)}&\leq c\|u_0\|_{H^{s_c}}+c\|u\|_{\tilde{S}_{2,\frac{2d}{d-2}}^{s_c}(I)}^2\|u\|_{\tilde{S}_{\infty,2}^{s_c}(I)}^{p-2}
\end{align*}
and
\begin{align*}
&\|\Phi_{u_0}(u)-\Phi_{u_0}(v)\|_{\tilde{S}_{2,\frac{2d}{d-2}}^0(I)}+\|\Phi_{u_0}(u)-\Phi_{u_0}(v)\|_{\tilde{S}_{\infty, 2}^0(I)}\\
&\leq c_0(\|u\|_{L_{t\in I}^{p-1}L_x^\infty}^{p-1}+\|v\|_{L_{t\in I}^{p-1}L_x^\infty}^{p-1})\|u-v\|_{L_{t\in I}^\infty L_x^2}\\
&\leq c(\|u\|_{\tilde{S}_{2,\frac{2d}{d-2}}^{s_c}(I)}^2\|u\|_{\tilde{S}_{\infty,2}^{s_c}(I)}^{p-2}+\|v\|_{\tilde{S}_{2,\frac{2d}{d-2}}^{s_c}(I)}^2\|v\|_{\tilde{S}_{\infty,2}^{s_c}(I)}^{p-2})\|u-v\|_{L_{t\in I}^\infty L_x^2}.
\end{align*}
Now we let $\delta=\delta(c,\|u_0\|_{H^{s_c}})>0$ be a sufficiently small number to be chosen later, and then we pick $T=T(u_0,\delta)>0$ such that
$$\|e^{it(-\Delta)^\sigma}u_0\|_{\tilde{S}_{2,\frac{2d}{d-2}}^{s_c}(I)}\leq\delta,$$
Define 
$$B=\Big\{u: \|u\|_{\tilde{S}_{2,\frac{2d}{d-2}}^{s_c}(I)}\leq 2\delta\textup{ and }\|u\|_{\tilde{S}_{\infty,2}^{s_c}(I)}\leq 2c\|u_0\|_{H^{s_c}}\Big\}$$
equipped with the norm
$$\|u\|_X:=\|u\|_{\tilde{S}_{2,\frac{2d}{d-2}}^{0}(I)}+\|u\|_{\tilde{S}_{\infty,2}^{0}(I)}.$$
Then, for $u\in B$, we have
\begin{align*}
\|\Phi_{u_0}(u)\|_{\tilde{S}_{2,\frac{2d}{d-2}}^{s_c}(I)}&\leq \delta+c (2\delta)^2(2c\|u_0\|_{{H}^{s_c}})^{p-2}\leq 2\delta,\\
\|\Phi_{u_0}(u)\|_{\tilde{S}_{\infty,2}^{s_c}(I)}&\leq c\|u_0\|_{{H}^{s_c}}+c (2\delta)^2(2c\|u_0\|_{{H}^{s_c}})^{p-2}\leq 2c\|u_0\|_{{H}^{s_c}}.
\end{align*}
Choosing sufficiently small $\delta>0$, we prove that $\Phi_{u_0}$ maps  $B$ to itself. Similarly, one can show 
$$\|\Phi_{u_0}(u)-\Phi_{u_0}(v)\|_X\leq\frac{1}{2}\|u-v\|_X.$$
Therefore, it follows that $\Phi_{u_0}$ is a contraction mapping in $B$.
\end{proof}

\begin{remark}
$(i)$ In the proofs, the $L_x^\infty$ norm bounds are crucial for the following reason. In Proposition \ref{Strichartz}, there is a loss of regularity except the trivial ones,
$$\|e^{it(-\Delta)^\sigma}u_0\|_{L_{t\in I}^\infty L_x^2}=\|u_0\|_{L^2}$$
and
$$\Big\|\int_0^t e^{it(-\Delta)^\sigma}F(s)ds\Big\|_{L_{t\in I}^\infty L_x^2}\leq \|F\|_{L_{t\in I}^1L_x^2}.$$
Hence, when we estimate the $L_{t\in I}^\infty H_x^s$ norm of the integral term in $\Phi_{u_0}(u)$, we are forced to use the trivial one
\begin{equation}
\Big\|\int_0^t e^{it(-\Delta)^\sigma}|u|^{p-1}u(s)ds\Big\|_{L_{t\in I}^\infty H_x^s}\leq \||u|^{p-1}u\|_{L_{t\in I}^1H_x^s}.
\end{equation}
Indeed, otherwise, we have a higher regularity norm on the right hand side. Then, we cannot close the contraction mapping argument. Moreover, if $u_0\in H^s$, there is no good bound for $\|e^{it(-\Delta)^\sigma}u_0\|_{L_{t\in I}^qW_x^{s,r}}$ except the trivial one $(q,r)=(\infty,2)$. Thus, we are forced to bound the right hand side of (3.10) by
$$\|u\|_{L_{t\in I}^{p-1}L_x^\infty}^{p-1}\|u\|_{L_{t\in I}^\infty H_x^s}.$$
Therefore, we should have a good control on $\|u\|_{L_{t\in I}^{p-1}L_x^\infty}$.\\
$(ii)$ When $p<3$, the $L_{t\in I}^{p-1}L_x^\infty$ norm is scaling-supercritical. Thus, based on our method, the assumptions on $p$ in Theorem \ref{subcritical LWP} and \ref{critical LWP} are optimal except $p=3$ in the critical case.
\end{remark}

\section{Small Data Scattering}

\begin{proof}[Proof of Theorem \ref{scattering}]
For simplicity, we consider the case $d\geq 3$ only. It follows from the estimates in the proof of Theorem \ref{critical LWP} that if $\|u_0\|_{H^s}$ is small enough, then
$$\|u(t)\|_{L_{t\in\mathbb{R}}^{p-1}L_x^\infty}+\|u(t)\|_{L_{t\in\mathbb{R}}^\infty H_x^{s_c}}\leq\|u(t)\|_{\tilde{S}_{2,\frac{2d}{d-2}}^{s_c}(\mathbb{R})}+\|u(t)\|_{\tilde{S}_{\infty,2}^{s_c}(\mathbb{R})}\lesssim\|u_0\|_{H^{s_c}}<\infty.$$
By Strichartz estimates, the fractional chain rule and \eqref{critical Lx-infty bound}, we prove that
\begin{align*}
&\|e^{-iT_1(-\Delta)^\sigma}u(T_1)-e^{-iT_2(-\Delta)^\sigma}u(T_2)\|_{H^{s_c}}\\
&=\Big\|\int_{T_1}^{T_2} e^{-is(-\Delta)^\sigma}(|u|^{p-1}u)(s)ds\Big\|_{H^{s_c}}\\
&\leq \|u(t)\|_{L_{t\in[T_1,T_2)}^{p-1}L_x^\infty}^{p-1}\|u(t)\|_{L_{t\in[T_1,T_2)}^\infty H_x^{s_c}}\to 0
\end{align*}
as $T_1,T_2\to\pm\infty$. Thus, the limits
$$u_\pm=\lim_{t\to\pm\infty} e^{-it(-\Delta)^\sigma}u(t)$$
exist in $H^{s_c}$. Repeating the above estimates, we show that
$$\|u(t)-e^{it(-\Delta)^\sigma}u_\pm\|_{H^{s_c}}=\|e^{-it(-\Delta)^\sigma}u(t)-u_\pm\|_{H^{s_c}}\to 0$$ 
as $t\to\pm\infty$.
\end{proof}

\section{Ill-posedness}

We will prove Theorem \ref{ill-posedness} following the strategy in \cite{CCT}. Throughout this section, we assume that $d=1,2$ or $3$ and $\frac{d}{4}<\sigma<1$. If $p$ is not an odd integer, we further assume that $p\geq k+1$, where $k$ is the smallest integer greater than $\frac{d}{2}$. 

First, we construct an almost non-dispersive solution by small dispersion analysis.

\begin{lemma}[Small dispersion analysis]\label{small dispersion analysis}
Given a Schwartz function $\phi_0$, let $\phi^{(\nu)}(t,x)$ be the solution to the fractional NLS
\begin{equation}\label{small dispersion}
i\partial_t u+\nu^{2\sigma}(-\Delta)^\sigma u+\mu|u|^{p-1}u=0,\ u(0)=\phi_0,
\end{equation}
and $\phi^{(0)}(t,x)$ be the solution to the ODE with no dispersion
$$i\partial_t u+\mu|u|^{p-1}u=0,\ u(0)=\phi_0,$$
that is,
\begin{equation}\label{no dispersion}
\phi^{(0)}(t,x)=\phi_0(x)e^{it\omega |\phi_0(x)|^{p-1}}.
\end{equation}
Then there exist $C, c>0$ such that if $0<\nu\leq c$ is sufficiently small, then
\begin{equation}\label{small dispersion estimate}
\|\phi^{(\nu)}(t)-\phi^{(0)}(t)\|_{H^k}\leq C\nu^{2\sigma}
\end{equation}
for all $|t|\leq c |\log \nu|^c$.
\end{lemma}

\begin{proof}
The proof closely follows the proof of Lemma 2.1 in \cite{CCT}.
\end{proof}

Obviously, $\phi^{(\nu)}(t,\nu x)$ is a solution to $(\textup{NLS}_\sigma)$. Moreover, $\phi^{(\nu)}(t,\nu x)$ is bounded and almost flat in the following sense.

\begin{corollary}\label{small dispersion corollary}
Let $\phi^{(\nu)}$, $\nu$ and $c$ be in Lemma \ref{small dispersion analysis}. Let $s\geq 0$. Then,
\begin{equation}\label{L^infty control}
\|\phi^{(\nu)}(t,\nu x)\|_{L_x^\infty}\sim 1
\end{equation}
and
\begin{equation}\label{H^s control}
\|\phi^{(\nu)}(t,\nu x)\|_{\dot{H}_x^s}\sim \nu^{s-\frac{d}{2}}(c |\log \nu|^c)^s
\end{equation}
for all $|t|\leq c |\log \nu|^c$.
\end{corollary}

\begin{proof}
Since $k>\frac{d}{2}$, by the Sobolev inequality, we have
\begin{align*}
\|\phi^{(\nu)}(t,\nu x)-\phi^{(0)}(t,\nu x)\|_{L_x^\infty}&=\|\phi^{(\nu)}(t, x)-\phi^{(0)}(t,x)\|_{L_x^\infty}\\
&\lesssim\|\phi^{(\nu)}(t)-\phi^{(0)}(t)\|_{H^k}\lesssim\nu^{2\sigma}.
\end{align*}
Then, \eqref{L^infty control} follows from the explicit formula \eqref{no dispersion} for $\phi^{(0)}(t,x)$. 
It follows from \eqref{small dispersion estimate} and \eqref{no dispersion} that
$$\|\phi^{(\nu)}(t, \nu x)\|_{\dot{H}_x^s}\leq \nu^{s-\frac{d}{2}}(\|\phi^{(0)}(t)\|_{\dot{H}^s}+\|\phi^{(\nu)}(t)-\phi^{(0)}(t)\|_{\dot{H}^s})\sim \nu^{s-\frac{d}{2}}(c |\log \nu|^c)^s.$$
\end{proof}

For $v\in\mathbb{R}^d$, we define the pseudo-Galilean transformation by
$$\mathcal{G}_vu(t,x)=e^{-iv\cdot x}e^{it|v|^{2\sigma}}u(t,x-2t\sigma|v|^{2(\sigma-1)}v).$$
Note that when $\sigma=1$, $\mathcal{G}_v$ is simply a Galilean transformation, and that NLS is invariant under this transformation, that is, if $u(t)$ solves NLS, so does $\mathcal{G}_vu(t)$. However, when $\sigma\neq1$, $(\textup{NLS}_\sigma)$ is not exactly symmetric with respect to pseudo-Galilean transformations. Indeed, if $u(t)$ solves $(\textup{NLS}_\sigma)$, then $\tilde{u}(t)=\mathcal{G}_vu(t)$ obeys $(\textup{NLS}_\sigma)$ with an error term
\begin{equation}\label{fNLS with error}
i\partial_t\tilde{u}+(-\Delta)^\sigma\tilde{u}+\omega|\tilde{u}|^{p-1}\tilde{u}=e^{it|v|^{2\sigma}}e^{-iv\cdot x}(\mathcal{E}u)(t,x-2\sigma t|v|^{2(\sigma-1)}v),
\end{equation}
where
$$\widehat{\mathcal{E}u}(\xi)=E(\xi)\hat{u}(\xi)$$
with 
$$E(\xi)=|\xi-v|^{2\sigma}-|\xi|^{2\sigma}-|v|^{2\sigma}+2\sigma|v|^{2(\sigma-1)}v\cdot\xi.$$
However, we note that
\begin{equation}\label{E bound}
|E(\xi)|\lesssim |\xi|^{2\sigma}.
\end{equation}
Indeed, if $|\xi|\leq\frac{|v|}{100}$, then
$$E(\xi)=\Big||v|^{2\sigma}\Big(|\tfrac{v}{|v|}-\tfrac{\xi}{|v|}|^{2\sigma}-1+2\sigma \tfrac{v}{|v|}\cdot\tfrac{\xi}{|v|}\Big)-|\xi|^{2\sigma}\Big|\lesssim|v|^{2\sigma}\tfrac{|\xi|^2}{|v|^2}+|\xi|^{2\sigma}\lesssim|\xi|^{2\sigma}.$$
Otherwise, 
$$E(\xi)\lesssim|\xi|^{2\sigma}+|v|^{2\sigma}+|\xi|^{2\sigma}+|v|^{2\sigma}+2\sigma|v|^{2\sigma-1}|\xi|\lesssim|\xi|^{2\sigma}.$$
Therefore, one would expect an \textit{almost} symmetry for an almost flat solution $u(t)$, such as $\phi^{(\nu)}(t,\nu x)$ in Lemma \ref{small dispersion analysis}. Precisely, we have the following lemma.

\begin{lemma}[Pseudo-Galilean transformation]\label{Pseudo-Galilean transformation}
Let $\phi^{(\nu)}$, $\nu$ and $c$ be in Lemma \ref{small dispersion analysis}. For $v\in\mathbb{R}^d$, we define
$$\tilde{u}(t,x)=(\mathcal{G}_v\phi^{(\nu)}(\cdot, \nu\cdot))(t,x)=e^{-iv\cdot x}e^{it|v|^{2\sigma}}\phi^{(\nu)}\big(t,\nu(x-2t\sigma|v|^{2(\sigma-1)}v)\big),$$
and let $u(t,x)$ be the solution to $(\textup{NLS}_\sigma)$ with the same initial data
\begin{equation}\label{initial data}
e^{-iv\cdot x}\phi^{(\nu)}(0,\nu x)=e^{-iv\cdot x}\phi_0(0,\nu x).
\end{equation}
Then, there exists $\delta>0$ such that
\begin{equation}\label{estim}
\|e^{iv\cdot x}(u(t)-\tilde{u}(t))\|_{H_x^k}\lesssim \nu^{\delta}
\end{equation}
for all $|t|\leq c |\log \nu|^c$.
\end{lemma}

\begin{remark}
When $p=3$, in \cite{CHKL} the authors could use the counterexample in  
\cite{CCTamer}. This counterexample is constructed by  
pseudo-conformal symmetry and Galilean transformation. A good thing is  
that this solution is very small in high Sobolev norms, too. Somehow,  
this smallness allows \cite{CHKL} to show that the error in pseudo-Galilean  
transformation is also small. However, when $p>3$, the counterexample in  
\cite{CCTamer} does not work. Later, Christ, Colliander and Tao \cite{CCT} constructed  
a different counterexample which works for more general $p$.  
Unfortunately, this counterexample is not small in high Sobolev norms.  
It is very large instead. In particular, for our purposes, it is hard to control the error from  
pseudo-Galilean transformation. But, our new  
counterexample still has small high Sobolev norm after translating it  
to its frequency center; this is the term $e^{iv\cdot x}$ in equation \eqref{estim}. Using this smallness, we can prove that pseudo-Galilean  
transformation is almost invariant. We also remark that the condition $\sigma>\frac{d}{4}$ is to guarantee smallness of the error (see \eqref{II(s) estimate}).
\end{remark}

\begin{proof}[Proof of Lemma \ref{Pseudo-Galilean transformation}]
Let $R(t)=(u-\tilde{u})(t)$. Then, $R(t)$ satisfies
$$i\partial_tR+(-\Delta)^\sigma R=\mu\big(|\tilde{u}|^{p-1}\tilde{u}-|u|^{p-1}u\big)-e^{it|v|^{2\sigma}}(\mathcal{E}\phi^{(\nu)}\big(t,\nu(x-2\sigma t|v|^{2(\sigma-1)}v)\big),$$
or equivalently
\begin{align*}
R(t)&=i\int_0^t e^{i(t-s)(-\Delta)^{\sigma}}\Big\{\mu\big(|u|^{p-1}u-|\tilde{u}|^{p-1}\tilde{u}\big)(s)\\
&\quad\quad\quad\quad\quad\quad\quad\quad+e^{is|v|^{2\sigma}}(\mathcal{E}\phi^{(\nu)}\big(s,\nu(x-2\sigma s|v|^{2(\sigma-1)}v)\big)\Big\}ds.
\end{align*}
Hence, by a trivial estimate, we get
\begin{align*}
\|e^{iv\cdot x}R(t)\|_{H^k}&\leq\int_0^t \big\|e^{iv\cdot x}\big(|u|^{p-1}u-|\tilde{u}|^{p-1}\tilde{u}\big)(s)\big\|_{H^k}+\|\mathcal{E}\phi^{(\nu)}(s,\nu\cdot)\|_{H^k}ds\\
&=\int_0^t I(s)+II(s)ds.
\end{align*}
First, by \eqref{E bound} and \eqref{H^s control}, we show that
\begin{equation}\label{II(s) estimate}
\int_0^tII(s)ds\lesssim \int_0^t\sum_{j=0}^k\|\phi^{(\nu)}(s,\nu\cdot)\|_{\dot{H}^{j+2\sigma}}ds\sim (c |\log \nu|^c)^{1+2\sigma-\frac{d}{2}}\nu^{2\sigma-\frac{d}{2}}.
\end{equation}
For $I(s)$, expanding $u=\tilde{u}+R$ and then applying H\"older inequality and Sobolev inequalities, we bound $I(s)$ by
\begin{equation}\label{I(s) bound}
\lesssim\sum_{j=1}^p\|e^{iv\cdot x}R\|_{H^k}^j.
\end{equation}
For example, when $p=3$,
\begin{align*}
I(s)&\leq2\| |e^{iv\cdot x}\tilde{u}|^2e^{iv\cdot x}R\|_{H^k}+\|(e^{iv\cdot x}\tilde{u})^2\overline{e^{iv\cdot x}R}\|_{H^k}+2\|e^{iv\cdot x}\tilde{u}|e^{iv\cdot x}R|^2\|_{H^k}\\
&\quad+\|\overline{e^{iv\cdot x}\tilde{u}}(e^{iv\cdot x}R)^2\|_{H^k}+\||e^{iv\cdot x}R|^2e^{iv\cdot x}R\|_{H^k}\\
&=:I_1(s)+I_2(s)+I_3(s)+I_4(s)+I_5(s).
\end{align*}
Consider
$$I_1(s)=\sum_{|\alpha|\leq k}\|\nabla_{x_1}^{\alpha_1}\cdots\nabla_{x_d}^{\alpha_d}(|e^{iv\cdot x}\tilde{u}|^2e^{iv\cdot x}R)(s)\|_{L^2}=:\sum_{|\alpha|\leq k}I_{1,\alpha}(s),$$
where $\alpha=(\alpha_1,\alpha_2,\cdots,\alpha_d)$ is a multi-index with $|\alpha|=\sum_{i=1}^d\alpha_i$. Observe that whenever a derivative hits
$$e^{iv\cdot x}\tilde{u}(s)=e^{is|v|^{2\sigma}}\phi^{(\nu)}\big(s,\nu(x-2s\sigma|v|^{2(\sigma-1)}v)\big),$$
we get a small factor $\nu$. Hence, after distributing derivatives by the Leibniz rule,  the worst term we have in $I_{1,\alpha}(s)$ is
$$\| |e^{iv\cdot x}\tilde{u}(s)|^2\nabla^\alpha e^{iv\cdot x}R(s)\|_{L^2},$$
which is, by \eqref{L^infty control}, bounded by 
$$\| e^{iv\cdot x}\tilde{u}(s)\|_{L^\infty}^2\|\nabla^\alpha e^{iv\cdot x}R(s)\|_{L^2}\sim \|\nabla^\alpha e^{iv\cdot x}R(s)\|_{L^2}.$$
Likewise, we estimate other terms.

Collecting all,
$$\|e^{iv\cdot x}R(t)\|_{H^k}\lesssim (c |\log \nu|^c)^{1+2\sigma-\frac{d}{2}}\nu^{2\sigma-\frac{d}{2}}+\int_0^t \sum_{j=1}^p\|e^{iv\cdot x}R(s)\|_{H^k}^jds$$
for $|t|\leq c|\log \nu|^c$. Then, by the standard nonlinear iteration argument, we prove the lemma.
\end{proof}

Since we have solutions almost symmetric with respect to the pseudo-Galilean transformations, we can make use of the following decoherence lemma to construct counterexamples for local well-posedness.

\begin{lemma}[Decoherence]\label{Decoherence}
Let $s<0$. Fix a nonzero Schwartz function $w$. For $a,a'\in[\frac{1}{2},1]$, $0<\nu\leq\lambda\ll 1$ and $v\in\mathbb{R}^d$ with $|v|\geq 1$, we define 
$$\tilde{u}^{(a,\nu,\lambda,v)}(t,x):=\mathcal{G}_v\Big(\lambda^{-\frac{2\sigma}{p-1}}\phi^{(a,\nu)}(\lambda^{-2\sigma}\cdot,\lambda^{-1}\nu\cdot)\Big)(t,x),$$
where $\phi^{(a,\nu)}$ is the solution to \eqref{small dispersion} with initial data $aw$. Then, we have
\begin{align*}
\|\tilde{u}^{(a,\nu,\lambda, v)}(0)\|_{H^s}, \|\tilde{u}^{(a',\nu,\lambda, v)}(0)\|_{H^s}&\leq C|v|^s\lambda^{-\frac{2\sigma}{p-1}}(\tfrac{\lambda}{\nu})^{d/2},\\
\|\tilde{u}^{(a,\nu,\lambda, v)}(0)-\tilde{u}^{(a',\nu,\lambda, v)}(0)\|_{H^s}&\leq C|v|^s\lambda^{-\frac{2\sigma}{p-1}}(\tfrac{\lambda}{\nu})^{d/2}|a-a'|
\end{align*}
and 
\begin{align*}
&\|\tilde{u}^{(a,\nu,\lambda, v)}(t)-\tilde{u}^{(a',\nu,\lambda, v)}(t)\|_{H^s}\\
&\geq c|v|^s\lambda^{-\frac{2\sigma}{p-1}}(\tfrac{\lambda}{\nu})^{d/2}\Big\{\|(\phi^{(a,\nu)}(\tfrac{t}{\lambda^{2\sigma}})-\phi^{(a',\nu)}(\tfrac{t}{\lambda^{2\sigma}})\|_{L^2}-C|\log \nu|^C(\tfrac{\lambda}{\nu})^{-k}|v|^{-s-k}\Big\}
\end{align*}
for all $|t|\leq c|\log \nu|^c\lambda^{2\sigma}$.
\end{lemma}

\begin{proof}
The proof closely follows the proof of Lemma 3.1 in \cite{CCT}.
\end{proof}

\begin{proof}[Proof of Theorem \ref{ill-posedness}]
The proof is very similar to that of Theorem 1 in \cite{CCT} except that in the last step, we need to use Lemma \ref{Pseudo-Galilean transformation} due to lack of exact symmetry. We give a proof for the readers' convenience.

Let $\epsilon>0$ be a given but arbitrarily small number. Let $\nu=\lambda^{\alpha}$, where $\alpha>0$ is a small number to be chosen later. Then, we pick $v\in \mathbb{R}^d$ such that
$$\lambda^{-\frac{2\sigma}{p-1}}|v|^s(\lambda/\nu)^{d/2}=\epsilon\Leftrightarrow |v|=\nu^{\frac{1}{s}(\frac{d(1-\alpha)}{2}+\frac{2\alpha\sigma}{p-1})}\epsilon^{1/s}.$$
Note that since $s<0$, $\frac{1}{s}(\frac{d(1-\alpha)}{2}+\frac{2\alpha\sigma}{p-1})=\frac{1}{s}(\frac{d}{2}-\alpha s_c)<0$ for sufficiently small $\alpha$, and thus $|v|\geq 1$. Hence, it follows from Lemma \ref{Decoherence} that
 \begin{align}
\|\tilde{u}^{(a,\nu,\lambda, v)}(0)\|_{H^s}, \|\tilde{u}^{(a',\nu,\lambda, v)}(0)\|_{H^s}&\leq C\epsilon,\\
\|\tilde{u}^{(a,\nu,\lambda, v)}(0)-\tilde{u}^{(a',\nu,\lambda, v)}(0)\|_{H^s}&\leq C\epsilon|a-a'|,
\end{align}
and 
\begin{align*}
&\|\tilde{u}^{(a,\nu,\lambda, v)}(t)-\tilde{u}^{(a',\nu,\lambda, v)}(t)\|_{H^s}\\
&\geq c\epsilon\Big\{\|(\phi^{(a,\nu)}(\tfrac{t}{\lambda^{2\sigma}})-\phi^{(a',\nu)}(\tfrac{t}{\lambda^{2\sigma}})\|_{L^2}-C|\log \nu|^C(\tfrac{\lambda}{\nu})^{-k}|v|^{-s-k}\Big\}
\end{align*}
for all $|t|\leq c|\log \nu|^c\lambda^{2\sigma}$. Now we observe from the explicit formula \eqref{no dispersion} for $\phi^{(a,0)}$ and \eqref{small dispersion} that there exists $T>0$ such that $\|\phi^{(a,\nu)}(T)-\phi^{(a',\nu)}(T)\|_{L^2}\geq c$. Moverover, if $\alpha>0$ is sufficiently small, $C|\log \nu|^C(\tfrac{\lambda}{\nu})^{-k}|v|^{-s-k}\to 0$ as $\nu \to0$. Therefore, for $\nu$ small enough, we have
\begin{equation}
\|\tilde{u}^{(a,\nu,\lambda, v)}(\lambda^{2\sigma}T)-\tilde{u}^{(a',\nu,\lambda, v)}(\lambda^{2\sigma}T)\|_{H^s}\geq c\epsilon.
\end{equation}
Next, we replace $\tilde{u}^{(a,\nu,\lambda, v)}$ and $\tilde{u}^{(a',\nu,\lambda, v)}$ in $(6.11)$, $(6.12)$ and $(6.13)$ by $u^{(a,\nu,\lambda, v)}$ and $u^{(a',\nu,\lambda, v)}$ by Lemma \ref{Pseudo-Galilean transformation} with $O(\nu^{\delta})$ erorr. Then, making $|a-a'|$ arbitrarily small and then sending $\nu\to 0$ (so, $\lambda^{2\sigma}T\to 0$), we complete the proof.
\end{proof}

\section*{Acknowledgements}
Y.H. would like to thank IH\'ES for their hospitality and support while he visited in the summer of 2014. Y.S. would like to thank the hospitality of the Department of Mathematics at University of Texas at Austin where part of the work was initiated. Y.S. acknowledges the support of ANR grants "HAB" and "NONLOCAL".  

\bibliographystyle{alpha} 
\bibliography{biblio}

\end{document}